\documentclass[12pt]{amsart} 

\usepackage{a4}
\usepackage{amsmath}
\usepackage{amssymb}
\usepackage{enumerate}
\usepackage[dvipdfmx]{hyperref}

\theoremstyle{definition}
\newtheorem{Def}{Definition}[section]

\theoremstyle{plain}
\newtheorem{Prop}[Def]{Proposition}
\newtheorem{Thm}[Def]{Theorem}
\newtheorem{Lem}[Def]{Lemma}
\newtheorem{Cor}[Def]{Corollary}

\theoremstyle{definition}
\newtheorem{Remark}[Def]{Remark}

\newcommand{\R}{\mathbb{R}} 
\newcommand{\LA}{\mathfrak{a}} 
\newcommand{\LG}{\mathfrak{g}} 
 
\newcommand{\LN}{\mathfrak{n}} 
\newcommand{\LS}{\mathfrak{s}} 

\newcommand{\inner}[2]{\langle #1 , #2 \rangle} 
\newcommand{\tr}{\mathop{\mathrm{tr}}\nolimits}

\numberwithin{equation}{section}

\def\rm#1{\mathrm{#1}}
\def\fr#1{\mathfrak{#1}}

\newcommand{\rank}{\mathop{\mathrm{rank}}\nolimits}
\newcommand{\ad}{\mathop{\mathrm{ad}}\nolimits}

\newcommand{\Der}{\mathop{\mathrm{Der}}\nolimits}
\newcommand{\Span}{\mathop{\mathrm{span}}\nolimits}

\newcommand{\Ric}{\mathop{\mathrm{Ric}}\nolimits}
\newcommand{\id}{\mathop{\mathrm{id}}\nolimits}

\newcommand{\Aut}{\mathop{\mathrm{Aut}}\nolimits}

\begin{document} 

\title[Realizations of some contact metric manifolds]{
Realizations of some contact metric manifolds as Ricci soliton real hypersurfaces
}

\author[J.~T.~Cho]{Jong Taek Cho}
\author[T.~Hashinaga]{Takahiro Hashinaga}
\author[A.~Kubo]{Akira Kubo}
\author[Y.~Taketomi]{Yuichiro Taketomi}
\author[H.~Tamaru]{Hiroshi Tamaru}
\address[J.~T.~Cho]{
Department of Mathematics, Chonnam National University, Gwangju 61186, Korea \\ 
}
\email{jtcho@chonnam.ac.kr}
\address[T.~Hashinaga]{
        National Institute of Technology, Kitakyushu College, 5-20-1 Shii, Kokuraminamiku, Kitakyushu, Fukuoka, 802-0985 Japan \\
        }
\email{hashinaga@kct.ac.jp}
\address[A.~Kubo]{
	Faculty of Economic Sciences, Hiroshima Shudo University, 
        Hiroshima 731-3195, Japan
        }
\email{akubo@shudo-u.ac.jp}
\address[Y.~Taketomi]{
	Department of Mathematics, Hiroshima University, 
        Higashi-Hiroshima 739-8526, Japan
        }
\email{y-taketomi@hiroshima-u.ac.jp}
\address[H.~Tamaru]{
	Department of Mathematics, Hiroshima University, 
        Higashi-Hiroshima 739-8526, Japan
        }
\email{tamaru@math.sci.hiroshima-u.ac.jp}

\thanks{
The first author was supported in part by Basic Science Research Program 
through the National Research Foundation of Korea (NRF) funded by 
the Ministry of Education, Science and Technology (2016R1D1A1B03930756). 
The second author was supported in part by JSPS KAKENHI (16K17603). 
The fifth author was supported in part by JSPS KAKENHI (26287012, 16K13757). 
}

\begin{abstract} 
Ricci soliton contact metric manifolds with certain nullity conditions 
have recently been studied by Ghosh and Sharma. 
Whereas the gradient case is well-understood, 
they provided a list of candidates for the nongradient case.
These candidates can be realized as Lie groups, 
but one only knows the structures of the underlying Lie algebras, 
which are hard to be analyzed apart from the three-dimensional case. 
In this paper, 
we study these Lie groups with dimension greater than three, 
and prove that the connected, simply-connected, and complete ones can be realized 
as homogeneous real hypersurfaces in noncompact real two-plane Grassmannians. 
These realizations enable us to prove, in a Lie-theoretic way, 
that all of them are actually Ricci soliton. 
\end{abstract} 

\maketitle 

\section{Introduction}

A class of contact metric manifolds with certain nullity conditions, 
so-called the $(\kappa, \mu)$-spaces, 
has been introduced by Blair, Koufogiorgos, and Papantoniou (\cite{BKP}) as follows. 

\begin{Def}
\label{def:kappa_mu}
Let $(\kappa, \mu) \in \mathbb{R}^2$. 
A contact metric manifold $(M, \eta, \xi, \varphi, g)$ is called a \textit{$(\kappa, \mu)$-space} 
if the Riemannian curvature tensor $R$ satisfies 
\begin{align} 
R(X,Y)\xi  = (\kappa I + \mu h) (\eta(Y) X - \eta(X) Y) 
\end{align} 
for any vector fields $X, Y \in \mathfrak{X}(M)$, where 
$I$ denotes the identity transformation and 
$h := (1/2)\mathcal{L}_\xi\varphi$ is the Lie derivative of $\varphi$ along $\xi$. 
\end{Def} 

The notion of $(\kappa,\mu)$-spaces is a generalization of Sasakian manifolds ($\kappa=1$ and $h=0$), 
and $(\kappa,\mu)$-spaces have fruitful geometric properties. 
Among others, $(\kappa,\mu)$-spaces have strongly pseudo-convex (integrable) CR-structure, 
and the class of $(\kappa, \mu)$-spaces is invariant under $D$-homothetic transformations. 
We refer to \cite{BKP} for details. 
Boeckx (\cite{Boe99}) has also studied $(\kappa, \mu)$-spaces deeply. 
He proved that every non-Sasakian $(\kappa,\mu)$-space is a locally homogeneous contact metric manifold, 
and its local geometry is completely determined by the dimension and the numbers $(\kappa, \mu)$.  

Recall that a Riemannian metric $g$ is called a \textit{Ricci soliton metric}, 
or just a \textit{Ricci soliton} for short, 
if there exist some constant $c \in \R$ and some vector field $X \in \mathfrak{X}(M)$ such that 
the Ricci tensor $\mathrm{Ric}_{g}$ satisfies 
\begin{align} 
\mathrm{Ric}_{g} = c g + \mathcal{L}_{X} g . 
\end{align} 
A Ricci soliton is said to be \textit{trivial} if it is Einstein, 
that is, $\mathcal{L}_{X} g = 0$ holds. 
Studies on Ricci solitons in the frame work of contact geometry are interesting and have been developed 
(see \cite{MR2735600, MR3333393, MR2434858} and references therein). 
For example, the first author has determined the structure of homogeneous contact gradient Ricci solitons (\cite{MR2764404}). 
In \cite{MR3333393}, 
Ghosh and Sharma have studied Ricci solitons on non-Sasakian $(\kappa, \mu)$-spaces. 

\begin{Thm}[\cite{MR3333393}] 
\label{thm:GS} 
Let $M$  be a non-Sasakian $(\kappa,\mu)$-space whose metric is a nontrivial Ricci soliton.
Then  $M$ is locally isometric to either $(0,0)$-space or $(0,4)$-space as contact metric manifolds. 
Especially, if $M$ is connected, simply-connected, and complete, 
then $M$ is isometric to one of the following spaces as contact metric manifolds$:$ 
\begin{enumerate}
  [$(1)$]
\item 
the three-dimensional Gaussian soliton $E^{3}$, 
\item 
the gradient shrinking rigid Ricci soliton $E^{n+1} \times S^{n}(4)$ with $n>1$, 
\item 
a simply-connected nongradient expanding Ricci soliton $(0,4)$-space. 
\end{enumerate}
\end{Thm} 

We refer to \cite{MR2265040} for details on Ricci solitons 
(some notions will be mentioned in Section~\ref{sec:RS}). 
In fact, the first two cases are well-known examples of gradient Ricci solitons (\cite{PW, MR2265040}), 
and are $(0,0)$-spaces. 

The third case, that is the case of $(0,4)$-spaces, is more complicated and quite interesting. 
Boeckx (\cite{Boe99}) constructed these spaces as Lie groups with left-invariant contact metric structures. 
According to his construction, 
the three-dimensional one is isometric to $\mathrm{Sol}$, 
in other words, 
the group $E(1,1)$ of rigid motions of the Minkowski two-space. 
This Ricci soliton has also been known (\cite{MR2604955}). 
Higher-dimensional $(0,4)$-spaces may be thought as a generalization of $\mathrm{Sol}$. 
For these spaces, 
one only knows the structures of the underlying Lie algebras, 
which are involved and hard to be analyzed. 

The main result of this paper gives realizations of $(0,4)$-spaces of dimension greater than three, 
as homogeneous real hypersurfaces 
in the noncompact real two-plane Grassmannians. 
Recall that a noncompact real two-plane Grassmannian $G_{2}^{\ast}(\mathbb{R}^{n+3})$ 
is a Hermitian symmetric space of noncompact type, which is K\"{a}hler, 
and a real hypersurface of a K\"{a}hler manifold admits the canonical almost contact metric structure 
(see Section~\ref{sec:contact} for details). 

\begin{Thm} 
\label{thm} 
Let $M$ be the connected, simply-connected, 
and complete $(0,4)$-space of dimension $2n+1$ with $n \geq 2$. 
Then $M$ is isomorphic to some homogeneous real hypersurface of $G_{2}^{\ast}(\mathbb{R}^{n+3})$ 
as contact metric manifolds. 
\end{Thm} 

This homogeneous real hypersurface has been constructed in \cite{MR2015258}, 
and studied in \cite{MR3326043, KT}. 
This hypersurface can be obtained as an orbit of some solvable Lie group, 
which is in fact a codimension one subgroup of the solvable part $AN$ 
of the Iwasawa decomposition $\mathrm{SO}_0(2,n+1) = KAN$. 
In view of these facts, 
one can apply a general theory of Ricci soliton solvmanifolds (see Section~\ref{sec:RS}), 
and a simple Lie-theoretic argument proves the following. 

\begin{Cor}
\label{cor} 
Every connected, simply-connected, and complete $(0,4)$-space is Ricci soliton. 
Therefore, a connected, simply-connected, and complete non-Sasakian $(\kappa, \mu)$-space 
is Ricci soliton if and only if $(\kappa, \mu) = (0,0)$ or $(0,4)$.  
\end{Cor} 

We note that, in \cite{MR3333393}, 
it has not been explicitly mentioned whether $(0,4)$-spaces of higher dimension are Ricci soliton or not. 
Therefore, our result complements the theorem of Ghosh and Sharma, 
from the viewpoint of submanifold geometry of symmetric spaces of noncompact type. 

\section{Notes on contact metric manifolds}
\label{sec:contact} 

In this section, 
we recall some necessary notions for contact metric manifolds, 
especially for $(\kappa,\mu)$-spaces. 
Let $M$ be a $(2n+1)$-dimensional manifold, 
and $\mathfrak{X}(M)$ denote the set of all smooth vector fields on $M$. 

\begin{Def} 
We call $M$ an \textit{almost contact manifold} if it is equipped with a $1$-form $\eta$, 
a vector field $\xi \in \mathfrak{X}(M)$, and a $(1,1)$-tensor field $\varphi$ such that 
\begin{align}\label{eq:contact1}
\eta(\xi) = 1, \quad \varphi^2 X = - X + \eta(X) \xi \quad (\forall X \in \mathfrak{X}(M)) . 
\end{align}
\end{Def} 

An almost contact manifold is denoted by a quadruplet $(M, \eta, \xi, \varphi)$. 
The vector field $\xi$ is called the \textit{characteristic vector field}. 
Note that one has 
\begin{align} \label{eq:contact2} 
\varphi \xi = 0, \quad \eta \circ \varphi = 0 . 
\end{align} 

\begin{Def} 
Let $(M, \eta, \xi, \varphi)$ be an almost contact manifold. 
Then, a Riemannian metric $g$ is called an \textit{associated metric} if it satisfies 
\begin{align} \label{eq:contact3}
g (\varphi X, \varphi Y) = g(X, Y) - \eta(X) \eta(Y) \quad (\forall X, Y \in \mathfrak{X}(M)) . 
\end{align}
\end{Def} 

We call such $(M, \eta, \xi, \varphi, g)$ an \textit{almost contact metric manifold}. 
Note that, for an almost contact manifold $(M, \eta, \xi, \varphi)$, 
there always exists an associated metric (see \cite{Sas}). 
It follows from (\ref{eq:contact1}), (\ref{eq:contact2}), and (\ref{eq:contact3}) that 
\begin{align} 
\eta(X) = g(X, \xi) \quad (\forall X \in \mathfrak{X}(M)) . 
\end{align} 

Typical examples of almost contact metric manifolds are given by real hypersurfaces in K\"{a}hler manifolds. 
Note that we need them in our main theorem. 

\begin{Prop}[{see \cite{MR3326043, B_lec}}] \label{ACMM} 
Let $(\overline{M}, J, g)$ be a K\"{a}hler manifold, 
and $M$ be a connected oriented real hypersurface in $\overline{M}$. 
Denote by $N$ a unit normal vector field of $M$. 
We define the structure $(\eta, \xi, \varphi, g)$ on $M$ as follows$:$ 
\begin{itemize}
  \item the Riemannian metric $g$ is induced from the Riemannian metric on $\overline{M}$,
  \item the characteristic vector field $\xi$ is defined by $\xi := -J N$, 
  \item the $1$-form $\eta$ is the metric dual of $\xi$, that is, $\eta(X) = g(X, \xi)$, 
  \item the $(1,1)$-tensor field $\varphi$ is defined by $\varphi X = J X - \eta (X) N$.
\end{itemize}
Then, $(M, \eta, \xi, \varphi, g)$ is an almost contact metric manifold.
\end{Prop}

For an almost contact metric manifold $(M, \eta, \xi, \varphi, g)$, 
the fundamental 2-form $\Phi$ on $M$ is defined by 
\begin{align} 
\Phi (X, Y) = g(X, \varphi Y) \quad (X, Y \in \mathfrak{X}(M)). 
\end{align}

\begin{Def}
An almost contact metric manifold $(M, \eta, \xi, \varphi, g)$ is called a 
\textit{contact metric manifold} 
if $\Phi = d\eta$ holds. 
\end{Def}

A contact metric manifold is denoted by $(M, \eta, \xi, \varphi, g)$, 
and $(\eta, \xi, \varphi, g)$ is called a \textit{contact metric structure} on $M$. 
Note that 
$\Phi = d\eta$ implies 
\begin{align} 
\eta \wedge (d\eta)^n \ne 0 , 
\end{align} 
where we recall $\dim M = 2n+1$. 

In this paper we study $(\kappa, \mu)$-spaces, 
defined in Definition~\ref{def:kappa_mu}. 
We here summarize some basic facts on them. 

\begin{Remark}[\cite{BKP}] 
\label{rem:BKP}
A $(\kappa , \mu)$-space satisfies $\kappa \le 1$. 
Furthermore, one has the following: 
\begin{enumerate} 
\item 
$\kappa = 1$ implies $h=0$. 
Therefore, a contact metric manifold is Sasakian if and only if 
it is a $(\kappa, \mu)$-space with $\kappa = 1$. 
\item 
Typical examples of non-Sasakian $(\kappa, \mu)$-spaces are given by 
the unit tangent sphere bundles 
over Riemannian manifolds of constant curvature $c \ne 1$, 
equipped with the natural contact metric structures. 
They are non-Sasakian $(\kappa, \mu)$-spaces with $\kappa = c(2-c)$ and $\mu = -2c$. 
\item 
The class of all $(\kappa, \mu)$-spaces are preserved by $D$-homothetic deformations. 
For a contact metric manifold $(M, \eta, \xi, \varphi, g)$, 
a $D$-homothetic deformation means the change of the structure tensors by 
\begin{align*} 
\bar{\eta} := a \eta , \quad 
\bar{\xi} := (1/a) \xi , \quad 
\bar{\varphi} := \varphi , \quad 
\bar{g} := a g + a (a-1) \eta \otimes \eta , 
\end{align*} 
where $a$ is a positive constant. 
Then, $(M, \bar{\eta}, \bar{\xi}, \bar{\varphi}, \bar{g})$ is a contact metric manifold, 
and a $D$-homothetic deformation maps a $(\kappa, \mu)$-space to a $(\bar{\kappa}, \bar{\mu})$-space, 
where 
\begin{align*} 
\bar{\kappa} = (\kappa + a^2 -1)/a^2 , \quad 
\bar{\mu} = (\mu + 2a -2)/a . 
\end{align*} 
\end{enumerate} 
\end{Remark} 

Non-Sasakian $(\kappa, \mu)$-spaces, that is the case of $\kappa < 1$, 
have deeply been studied by Boeckx (\cite{Boe}). 
Among others, the following local rigidity is fundamental. 

\begin{Thm}[{\cite[Theorem~3]{Boe}}] 
\label{thm:boecks} 
Let $(M, \eta, \xi, \varphi, g)$ and $(M', \eta', \xi', \varphi', g')$ be 
$(2n+1)$-dimensional non-Sasakian $(\kappa, \mu)$-spaces. 
Then they are locally isometric as contact metric manifolds. 
In particular, if both are connected, simply-connected, and complete, then they are globally isometric. 
\end{Thm} 

This theorem enables one to classify 
connected, simply-connected, and complete non-Sasakian $(\kappa, \mu)$-spaces. 
In fact, these spaces can be divided into two classes. 
The first class consists of the unit tangent sphere bundles over spaces of constant curvature, 
mentioned in Remark~\ref{rem:BKP}. 
The second class is given in \cite{Boe} 
as Lie groups equipped with suitable contact metric structures. 

\begin{Def} \label{def:gab}
For each $\alpha, \beta \in \R$, 
we define a real $(2n+1)$-dimensional Lie algebra $\fr{g}_{\alpha, \beta}$ 
with basis $\{\xi, X_1, X_2, \ldots, X_{n}, Y_1, Y_2, \ldots, Y_{n}\}$ as follows: 
\begin{enumerate}[$(1)$]
\item 
the bracket product $[\xi , X_i]$ is given by 
   	$$\begin{aligned}[]
        [\xi, X_1] &= - (1/2) \alpha \beta X_2 - (1/2) \alpha^2 Y_1, \\
        [\xi, X_2] &= (1/2) \alpha \beta X_1 - (1/2) \alpha^2 Y_2, \\
        [\xi, X_i] &= - (1/2) \alpha^2 Y_i & (i \ne 1,2) , 
        \end{aligned}$$
\item 
the bracket product $[\xi , Y_i]$ is given by 
   	$$\begin{aligned}[]
        [\xi, Y_1] &= (1/2) \beta^2 X_1 - (1/2) \alpha \beta Y_2, \\
        [\xi, Y_2] &= (1/2) \beta^2 X_2 + (1/2) \alpha \beta Y_1, \\
        [\xi, Y_i] &= (1/2) \beta^2 X_i & (i \ne 1,2) , 
        \end{aligned}$$
\item 
the bracket product $[X_i , X_j]$ is given by 
   	$$\begin{aligned}[]
        [X_1, X_i] &= \alpha X_i & (i \ne 1) , \\
        [X_i, X_j] &= 0 & (i,j \ne 1) , 
        \end{aligned}$$
\item 
the bracket product $[Y_i , Y_j]$ is given by 
   	$$\begin{aligned}[]
        [Y_2, Y_i] &= \beta Y_i & (i \ne 2) , \\
        [Y_i, Y_j] &= 0 & (i,j \ne 2) , 
        \end{aligned}$$
\item 
the bracket product $[X_i , Y_j]$ is given by 
   	$$\begin{aligned}[]
        [X_1, Y_1] &= -\beta X_2 + 2 \xi, \\
        [X_1, Y_i] &= 0 & (i \ne 1) , \\
        [X_2, Y_1] &= \beta X_1 - \alpha Y_2, \\
        [X_2, Y_2] &= \alpha Y_1 + 2 \xi, \\
        [X_2, Y_i] &= \beta X_i & (i \ne 1,2) , \\
        [X_i, Y_1] &= -\alpha Y_i & (i \ne 1,2) , \\
        [X_i, Y_2] &= 0 & (i \ne 1,2) , \\
        [X_i, Y_j] &= \delta_{i j}(-\beta X_2+\alpha Y_1 + 2\xi) & (i,j \ne 1,2) . 
        \end{aligned}$$
\end{enumerate}
\end{Def}

Let $G_{\alpha, \beta}$ be the simply-connected Lie group with Lie algebra $\fr{g}_{\alpha, \beta}$, 
and define left-invariant structures on $G_{\alpha, \beta}$ as follows: 
\begin{itemize}
  \item the left-invariant metric $g$ is defined so that the above basis is orthonormal,
  \item the characteristic vector field is given by $\xi$,
  \item the $1$-form $\eta$ is the metric dual of $\xi$, that is, $\eta(X) = g(X, \xi)$, 
  \item the $(1,1)$-tensor field $\varphi$ is defined by
\begin{align}
	\varphi(\xi)=0, \quad \varphi(X_i) = Y_i, \quad \varphi(Y_i)=-X_i.
\end{align}
\end{itemize}
Then, Boeckx has shown the following. 

\begin{Prop}[{\cite[Section~4]{Boe}}] 
\label{prop:G-alpha-beta} 
Assume that $\beta > \alpha \ge 0$. 
Then, the contact metric manifold $(G_{\alpha,\beta}, \eta, \xi, \varphi, g)$ 
is a non-Sasakian $(\kappa,\mu)$-space, where 
\begin{align} 
\kappa = 1 - \frac{(\beta^2-\alpha^2)^2}{16}, \qquad \mu = 2 + \frac{\alpha^2+\beta^2}{2} . 
\end{align} 
\end{Prop} 

In this paper, we are interested in $(0,4)$-spaces. 
Note that $(\alpha, \beta) = (0,2)$ implies $(\kappa, \mu) = (0,4)$. 
Therefore, 
it then follows from Theorem~\ref{thm:boecks} that 
a connected, simply-connected, and complete $(0,4)$-space is isomorphic to $G_{0,2}$. 
This Lie group $G_{0,2}$ will be studied in the following sections. 

\section{Notes on Ricci soliton solvmanifolds} 
\label{sec:RS} 

This section gives some preliminaries for left-invariant 
Einstein or Ricci soliton metrics on solvable Lie groups. 
The results of this section will be used 
to show that the $(0,4)$-spaces $G_{0,2}$ are nontrivial Ricci soliton. 

\subsection{Einstein solvmanifolds} 

A metric Lie algebra $(\LG , \inner{}{})$ is said to be \textit{Einstein} 
if the Ricci operator satisfies $\Ric = c I$ for some $c \in \R$. 
Note that a metric Lie algebra is Einstein if and only if 
the induced left-invariant metric on a corresponding Lie group is Einstein. 
A solvable Lie group equipped with a left-invariant Einstein metric is called an Einstein solvmanifold. 
In order to recall a general theory of Einstein solvmanifolds, 
we need the following notion. 

\begin{Def} 
A metric solvable Lie algebra $(\LS, \langle , \rangle)$ is said to be of \textit{Iwasawa type} if 
\begin{enumerate} 
\item 
$\LA := \LS \ominus [\LS, \LS]$ is abelian, 
\item 
$\ad_A$ ($A \in \LA$) is symmetric with respect to $\langle, \rangle$, and $\ad_A \neq 0$ if $A \neq 0$, 
\item 
there exists $A_0 \in \LA$ such that $\ad_{A_0} |_{[\LS, \LS]}$ is positive definite. 
\end{enumerate} 
\end{Def} 

Note that ${\mathfrak u} \ominus {\mathfrak u}^\prime$ 
denotes the orthogonal complement of ${\mathfrak u}^\prime$ in ${\mathfrak u}$. 
In order to study the Einstein condition on a metric Lie algebra, 
the following particular vector plays an important role. 

\begin{Def} 
For a metric Lie algebra $(\LG, \langle, \rangle)$, 
a vector $H_0 \in \LG$ is called the \textit{mean curvature vector} if it satisfies 
\begin{align*} 
\langle H_0 , X \rangle = \tr (\ad_X) \quad (\forall X \in \LG) . 
\end{align*} 
\end{Def} 

Heber (\cite{Heb}) deeply studied Einstein solvmanifolds, 
and obtained many fundamental results. 
Among others, he obtained the following reduction procedure, called the rank reduction. 

\begin{Thm}[{\cite[Theorem~4.18]{Heb}}] 
\label{thm:heber} 
Let $(\LS, \langle , \rangle)$ be an Einstein solvable Lie algebra of Iwasawa type, 
and $H_0$ be the mean curvature vector of $(\LS, \langle , \rangle)$. 
We put $\LN := [\LS, \LS]$, $\LA := \LS \ominus \LN$, 
and take a nonzero subspace $\LA^\prime \subset \LA$. 
Then $\LS^\prime := \LA^\prime \oplus \LN$ is a subalgebra. 
Furthermore, $(\LS^\prime , \langle , \rangle |_{\LS^\prime \times \LS^\prime})$ 
is Einstein if and only if $H_0 \in \LA^\prime$. 
\end{Thm} 

This gives a simple and practical criterion, 
for a particular type of subgroups of an Einstein solvmanifold of Iwasawa type, 
to be Einstein. 

\subsection{Ricci soliton solvmanifolds} 

A solvable Lie group equipped with a left-invariant Ricci soliton metric is called a 
Ricci soliton solvmanifold. 
The following notion plays a key role in the study of left-invariant Ricci soliton metrics. 

\begin{Def} 
A metric Lie algebra $(\LG , \inner{}{})$ is called an 
\textit{algebraic Ricci soliton with constant $c \in \R$} 
if there exists a derivation $D \in \Der(\LG)$ such that 
\begin{align*} 
\Ric = c I + D . 
\end{align*} 
\end{Def} 

Note that any algebraic Ricci soliton gives rise to a Ricci soliton metric 
on the corresponding simply-connected Lie group $G$ (see \cite{L01, L11}). 
In fact, one can express the potential vector field $X \in \mathfrak{X}(G)$ explicitly by 
\begin{align} 
\textstyle 
X_p := \frac{d}{dt} \varphi_t(p) |_{t=0} \quad (p \in G) , 
\end{align} 
where $\varphi_t \in \Aut(G)$ is defined by $(d \varphi_t)_e = e^{- t D} \in \Aut(\LG)$. 
The obtained left-invariant Ricci soliton metrics are also called 
algebraic Ricci solitons. 

An algebraic Ricci soliton is called a \textit{solvsoliton} if $\LG$ is solvable, 
and a \textit{nilsoliton} if $\LG$ is nilpotent. 
A solvsoliton is said to be \textit{nontrivial} if it is not Einstein. 
For nontrivial solvsolitons, 
the following properties have been known. 

\begin{Prop}[\cite{L11}] 
\label{prop:L11} 
Let $(\LS, \langle, \rangle)$ be a nontrivial solvsoliton. 
Then the corresponding left-invariant metric on the simply-connected Lie group 
is nongradient and expanding Ricci soliton. 
\end{Prop} 

Recall that a Ricci soliton is said to be 
\textit{expanding}, \textit{steady}, or \textit{shrinking}, 
according as $c<0$, $c=0$, or $c>0$, respectively. 
A Ricci soliton is said to be \textit{gradient} 
if the vector field $X$ is the gradient $\nabla f$ of a smooth function $f$. 

Many structural results of solvsolitons have been obtained by Lauret (\cite{L11}). 
Among others, solvsolitons can be characterized in terms of the nilradicals. 

\begin{Thm}[{\cite[Theorem~4.8]{L11}}] 
\label{thm:lauret} 
Let $(\LS , \inner{}{})$ be a solvable metric Lie algebra, 
and $\LN$ be the nilradical of $\LS$. 
We put $\LA := \LS \ominus \LN$, and take $c<0$. 
Then, $(\LS , \inner{}{})$ is a solvsoliton with constant $c$ if and only if 
the following conditions hold$:$ 
\begin{enumerate}
\item[$(1)$] 
$(\LN , \inner{}{}|_{\LN \times \LN})$ is a nilsoliton with constant $c$, 
\item[$(2)$] 
$[\LA , \LA] = 0$, 
\item[$(3)$] 
$[\ad_A , (\ad_A)^t] = 0$ for all $A \in \LA$, 
\item[$(4)$] 
$\inner{A}{A} = - (1/c) \tr S(\ad_A)^2$ for all $A \in \LA$. 
\end{enumerate}
\end{Thm}

Note that $(\ad_A)^t$ denotes the transpose of $\ad_A$ with respect to $\inner{}{}$, 
which is defined by 
\begin{align} 
\inner{(\ad_A)^t (X)}{Y} = \inner{X}{\ad_A(Y)} \quad (\forall X, Y \in \LS) . 
\end{align} 
We also note that $S (\ad_A)$ denotes the symmetrizer of $\ad_A$, that is, 
\begin{align} 
S (\ad_A) := (1/2) \left( \ad_A + (\ad_A)^t \right) . 
\end{align} 

As a direct corollary of Theorem~\ref{thm:lauret}, 
we have the following ``rank reduction'' of solvsolitons. 
In this case, one can take any subspace $\LA^\prime$ of $\LA$. 

\begin{Prop} 
\label{prop:solvsoliton} 
Let $(\LS , \inner{}{})$ be a solvsoliton with constant $c<0$, 
$\LN$ be the nilradical of $\LS$, and $\LA := \LS \ominus \LN$. 
Take any subspace $\LA'$ of $\LA$, and put $\LS' := \LA' \oplus \LN$. 
Then, $\LS'$ is a subalgebra, and 
$(\LS' , \inner{}{}|_{\LS' \times \LS'})$ is a solvsoliton with constant $c$. 
\end{Prop} 

\begin{proof} 
It is clear that $\LS'$ is a subalgebra of $\LS$. 
We show the second assertion. 
Since $(\LS , \inner{}{})$ is a solvsoliton, 
it satisfies Conditions~(1)--(4) in Theorem~\ref{thm:lauret}. 
We show that $(\LS' , \inner{}{}|_{\LS' \times \LS'})$ also satisfies these four conditions. 

First of all, it satisfies Condition~(1), since $\LN$ is the nilradical of $\LS'$. 
It also satisfies Condition~(2), since $[\LA' , \LA'] \subset [\LA , \LA] = 0$. 
We here denote by $\ad'$ the adjoint operator of $\LS'$. 
Since $\LS'$ is a subalgebra of $\LS$, one can show that 
\begin{align} 
\label{eq:lauret} 
\ad'_X = \ad_X |_{\LS'} , \ (\ad'_X)^t = (\ad_X)^t |_{\LS'} \quad (\forall X \in \LS') . 
\end{align} 
Take any $A \in \LA'$. 
Then, since $\ad_A$ and $(\ad_A)^t$ commute by assumption, we have 
\begin{align} 
[ \ad'_A , (\ad'_A)^t ] = [ \ad_A |_{\LS'} , (\ad_A)^t |_{\LS'} ] = 0 . 
\end{align} 
This shows Condition~(3). 
It remains to prove Condition~(4). 
By assumption, one has only to show that 
\begin{align} 
\label{eq:condition4} 
\tr S(\ad'_A)^2 = \tr S(\ad_A)^2 . 
\end{align} 
Consider the orthogonal decomposition $\LS = \LS' \oplus (\LS \ominus \LS')$. 
Note that $\LS \ominus \LS' \subset \LA$. 
Hence, one can see that $\ad_A$ and $(\ad_A)^t$ normalize $\LS'$, 
and act trivially on $\LS \ominus \LS'$. 
This yields that 
\begin{align} 
\begin{split} 
\tr \ad_A & = \tr \ad_A |_{\LS'} = \tr \ad'_A , \\  
\tr (\ad_A)^t & = \tr (\ad_A)^t |_{\LS'} = \tr (\ad'_A)^t . 
\end{split} 
\end{align} 
This completes the proof of \eqref{eq:condition4}. 
\end{proof} 


\section{The solvable model of the noncompact real two-plane Grassmannian} 
\label{sec:4} 

In this section, 
we introduce the \textit{solvable model} of the noncompact real two-plane Grassmannian 
\begin{align} 
G_2^\ast(\R^{n+2}) = \rm{SO}_0(2,n)/(\rm{SO}(2)\times \rm{SO}(n)) , 
\end{align} 
with $n \ge 3$. 
Note that $M = G_2^\ast(\R^{n+2})$ is an irreducible Hermitian symmetric space of noncompact type 
with $\rank M = 2$ and $\dim M = 2n$. 

\begin{Def}\label{solv.model}
Let $c>0$ and $n \ge 3$. We call $(\fr{s}(c), \langle, \rangle, J)$ 
the \textit{solvable model} of $G_2^\ast(\R^{n+2})$ if
\begin{enumerate}[$(1)$]
  \item $\fr{s}(c) :=
  	 \Span \{A_1, A_2, X_0, Y_1, \ldots, Y_{n-2}, Z_1, \ldots, Z_{n-2}, W_0\}$
          is a $2n$-dimensional Lie algebra whose bracket relations are defined by
  \begin{itemize}
  \item $[A_1, X_0] = c X_0$, $[A_1, Y_i] = -(c/2) Y_i$, $[A_1, Z_i] = (c/2) Z_i$, $[A_1, W_0] = 0$, 
  \item $[A_2, X_0] = 0  $,   $[A_2, Y_i] =  (c/2) Y_i$, $[A_2, Z_i] = (c/2) Z_i$, $[A_2, W_0] = c W_0$, 
  \item $[X_0, Y_i] = c Z_i$,   $[Y_i, Z_i] = c W_0$,
  \item and other relations vanish,
  \end{itemize}
  \item $\langle, \rangle$ is an inner product on $\fr{s}(c)$
so that the above basis is orthonormal, 
  \item $J$ is a complex structure on $\fr{s}(c)$ 
given by
\begin{align*}
J(A_1) = -X_0 , \quad J(A_2) =  W_0 , \quad J(Y_i) =  Z_i . 
\end{align*}
\end{enumerate}
\end{Def}

Let $S(c)$ denote the simply-connected Lie group with Lie algebra $\fr{s}(c)$. 
The induced left-invariant metric and (almost) complex structure on $S(c)$ 
are denoted by the same symbols $\langle, \rangle$ and $J$, respectively. 
Our aim in this section is to prove the following theorem. 

\begin{Thm}\label{thm_main4}
The Lie group $S(c)$ equipped with the left-invariant metric $\langle, \rangle$ and the complex structure $J$ 
is isomorphic to $G_2^\ast(\R^{n+2})$ with minimal sectional curvature $-c^2$.
\end{Thm}

\subsection{The Iwasawa decompositions} 

In this subsection, 
we recall some general facts on Iwasawa decompositions and the solvable parts.  
Refer to \cite{Hel}. 

Let $M=G/K$ be an irreducible Riemannian symmetric space of noncompact type, 
where $G$ is the identity component of the isometry group, 
and $K$ is the isotropy subgroup of $G$ at some point $o$, called the \textit{origin}. 
Denote by $(\fr{g} , \fr{k} , \theta)$ the corresponding symmetric pair, where 
$\fr{g}$ and $\fr{k}$ are the Lie algebras of $G$ and $K$ respectively, and $\theta$ is the Cartan involution. 
The eigenspace decomposition $\fr{g} = \fr{k} \oplus \fr{p}$ with respect to $\theta$ 
is called the \textit{Cartan decomposition}. 

In order to define the Iwasawa decomposition, 
we need the restricted root system. 
Let us fix $\fr{a}$ as a maximal abelian subspace of $\fr{p}$, 
and denote by $\fr{a}^\ast$ the dual space of $\fr{a}$.
Then we define
\begin{align}
\fr{g}_\lambda := \{ X \in \fr{g} \mid \ad(H)X = \lambda(H) X \ 
(\forall H \in \fr{a}) 
\}
\end{align}
for each $\lambda \in \fr{a}^\ast$. 
We call $\lambda \in \fr{a}^\ast$ a \textit{restricted root} with respect to $\fr{a}$ 
if $\fr{g}_\lambda \ne 0$ and $\lambda \neq 0$. 
Denote by $\Sigma$ the set of restricted roots. 
Let $\Lambda$ be a set of simple roots of $\Sigma$, 
and then denote by $\Sigma^+$ the set of positive roots associated with $\Lambda$. 
Let us define 
\begin{align}
\fr{n} := \bigoplus_{\lambda \in \Sigma^+} \fr{g}_\lambda . 
\end{align} 
Since $[\fr{g}_\lambda, \fr{g}_\mu] \subset \fr{g}_{\lambda+\mu}$ holds 
for all $\lambda , \mu \in \Sigma \cup \{ 0 \}$, 
we have that $\fr{n}$ is a nilpotent Lie subalgebra, and 
$\fr{a} \oplus \fr{n}$ is a solvable Lie subalgebra of $\fr{g}$. 

\begin{Def} 
The decomposition 
$\fr{g} = \fr{k} \oplus \fr{a} \oplus \fr{n}$ 
is called the \textit{Iwasawa decomposition}, 
and $\fr{a} \oplus \fr{n}$ 
is called the \textit{solvable part} of the Iwasawa decomposition. 
\end{Def} 

The solvable part $\fr{a} \oplus \fr{n}$ plays a crucial role in the study of the symmetric space $M = G/K$. 
The reason is the following. 
We refer to \cite[Ch.~VI]{Hel}. 

\begin{Prop} 
\label{prop:S-diffeo} 
Let $S$ be the connected Lie subgroup of $G$ with Lie algebra $\fr{a} \oplus \fr{n}$. 
Then $S$ is simply-connected and acts simply transitively on $M = G/K$. 
In particular, the following map is a diffeomorphism$:$ 
\begin{align}
\Phi : S \to G/K=M : g \mapsto [g]=g.o . 
\end{align}
\end{Prop} 

\subsection{The solvable models} 

Let $M = G/K$ be an irreducible Hermitian symmetric space of noncompact type. 
In this subsection, we recall the definition of its solvable model, 
which consists of the solvable part of the Iwasawa decomposition 
and two geometric structures on it. 

First of all, we review geometric structures on $M = G/K$. 
Recall that one can identify $\fr{p} \cong T_o(G/K)$. 
Hence, every $K$-invariant geometric structures on $\fr{p}$ 
(such as inner products and complex structures) 
can be extended to $G$-invariant geometric structures on $M = G/K$. 
The following are well-known facts on Hermitian symmetric spaces. 

\begin{Prop} 
\label{prop:structure_general} 
Let $M = G/K$ be an irreducible Hermitian symmetric space of noncompact type. 
Then we have the following$:$ 
\begin{enumerate} 
\item[$(1)$] 
Let $B$ be the Killing form of $\fr{g}$, and define 
\begin{align*} 
\langle X, Y \rangle_{\fr{g}} := - k B(\theta X, Y) 
\end{align*} 
with $k > 0$. 
Then $\langle , \rangle_{\fr{g}} |_{\fr{p} \times \fr{p}}$ 
gives a $G$-invariant metric on $M = G/K$, 
and every $G$-invariant metric on $M = G/K$ can be obtained in this way. 
\item[$(2)$] 
Let $C(\fr{k})$ be the center of $\fr{k}$. 
Then there exists $Z \in C(\fr{k})$ such that 
\begin{align*} 
(\ad(Z)|_{\fr{p}})^2 = - \id_{\fr{p}} . 
\end{align*} 
This gives a complex structure on $M = G/K$, which is unique up to sign. 
\end{enumerate} 
\end{Prop} 

\begin{proof} 
One can easily check that 
$\langle , \rangle_{\fr{g}} |_{\fr{p} \times \fr{p}}$ and $\ad(Z)|_{\fr{p}}$ are $K$-invariant. 
Hence they give rise to a $G$-invariant Riemannian metric and a $G$-invariant complex structure on $M = G/K$, 
respectively. 
The uniqueness of the metric follows from the Schur lemma, 
since $M = G/K$ is irreducible. 
For the uniqueness of the complex structure, 
see \cite[Theorem~4.5, Ch.~VIII]{Hel}. 
\end{proof} 

We here define the solvable model of $M = G/K$. 
In order to give geometric structures involved, we need the linear map 
\begin{align}
\pi : \fr{a} \oplus \fr{n} \to \fr{p}  : X \mapsto X_\fr{p} = (1/2) (X - \theta X) , 
\end{align}
where $\fr{p}$-subscript means the orthogonal projection onto $\fr{p}$. 
Note that $\pi$ is a linear isomorphism, since it satisfies 
\begin{align} 
\label{eq:pi} 
\pi = (d \Phi)_e , 
\end{align} 
where $\Phi : S \to G/K$ is the diffeomorphism given in Proposition~\ref{prop:S-diffeo}. 
Note that Equation~(\ref{eq:pi}) can be found in the proof of \cite[Proposition~4.4]{Tamaru}. 

\begin{Def} 
\label{def:model} 
Let $\fr{a} \oplus \fr{n}$ be the solvable part of the Iwasawa decomposition. 
Then the triplet $(\fr{a} \oplus \fr{n} , \langle, \rangle, J)$ defined by the following is called the 
\textit{solvable model} 
of $M = G/K$$:$ 
\begin{align*}
\langle X, Y \rangle
& := \langle X_\fr{a}, Y_\fr{a} \rangle_{\fr{g}} + (1/2) \langle X_\fr{n}, Y_\fr{n} \rangle_{\fr{g}} , \\ 
J & := \pi^{-1} \circ \ad(Z) \circ \pi . 
\end{align*}
\end{Def} 

In the following, 
we denote the induced left-invariant metric and complex structure on $S$ 
by the same symbols $\langle , \rangle$ and $J$, respectively. 
The following is a key fact on the solvable model. 

\begin{Prop} 
\label{prop:model_isomorphic} 
Assume that $M = G/K$ is equipped with the Riemannian metric given by 
$\langle , \rangle_{\fr{g}} |_{\fr{p} \times \fr{p}}$, 
and with the complex structure defined by $\ad(Z)|_{\fr{p}}$, 
described in Proposition~\ref{prop:structure_general}. 
Let $(\fr{a} \oplus \fr{n} , \langle, \rangle, J)$ be the solvable model. 
Then the corresponding $(S , \langle, \rangle, J)$ is isomorphic to $M = G/K$. 
\end{Prop} 

\begin{proof} 
Recall that $\Phi : S \to G/K$ is a diffeomorphism. 
It follows from (\ref{eq:pi}) that $\Phi$ preserves the metrics 
(see \cite[Proposition~4.4]{Tamaru}), 
and also the complex structures. 
\end{proof} 

\subsection{The Iwasawa decomposition of $\mathfrak{so}(2,n)$} 

We now study the noncompact real two-plane Grassmannian 
\begin{align} 
G_2^\ast(\R^{n+2}) = \rm{SO}_0(2,n)/(\rm{SO}(2)\times\rm{SO}(n)) , 
\end{align} 
with $n \ge 3$. 
Let $\fr{g}$ and $\fr{k}$ 
denote the Lie algebras of $\rm{SO}_0(2,n)$ and $\rm{SO}(2)\times\rm{SO}(n)$, respectively. 
In this subsection, we describe the solvable part of the Iwasawa decomposition of $\fr{g} = \fr{so}(2,n)$. 

First of all, we recall the Cartan decomposition of $\fr{g}$. 
Here and hereafter, we denote by $E_{i,j}$ (or $E_{i j}$ shortly) the usual matrix unit, and put 
\begin{align}
I_{2,n} := - E_{11} - E_{22} + E_{33} + \cdots + E_{n+2,n+2}.
\end{align}
Then one knows that 
\begin{align} 
\fr{g} &= 
\{X \in \fr{sl}(n+2, \R) \mid {}^t X I_{2,n}  + I_{2,n} X = 0 \}, \\
\fr{k} &= \left\{ \left(\begin{array}{c|ccc}
		X & &   & \\\hline
		  & &   & \\
		  & & Y & \\
		  & &   & 
		\end{array} \right) \mid X \in \fr{so}(2), Y \in \fr{so}(n) \right\} , 
\end{align} 
and the Cartan involution $\theta$ is given by 
\begin{align}
\theta : \fr{g} \to \fr{g} : X \mapsto I_{2,n} X I_{2,n} . 
\end{align}
The Cartan decomposition $\fr{g} = \fr{k} \oplus \fr{p}$ with respect to $\theta$ is given by 
\begin{align}
\fr{p} := \left\{ \left(\begin{array}{c|ccc}
		       & & v & \\\hline
		       & &   & \\
		{}^t v & &   & \\
		       & &   & 
		\end{array} \right) \mid v \in M_{2,n}(\R) \right\}.
\end{align}

We next describe the restricted root system and the root spaces of the symmetric space $G_2^\ast(\R^{n+2})$. 
Let us put 
\begin{align}
\fr{a} := \Span\{E_{13}+E_{31}, E_{24} + E_{42}\} \subset \fr{g} , 
\end{align}
which is a maximal abelian subspace in $\fr{p}$. 
Define $\varepsilon_i \in \fr{a}^\ast$ $(i=1,2)$ by 
\begin{align}
\varepsilon_i : \fr{a} \to \R : a_1 (E_{13}+E_{31}) + a_2 (E_{24} + E_{42}) \mapsto a_i . 
\end{align}
We put $\alpha_1 := \varepsilon_1 - \varepsilon_2$, $\alpha_2 := \varepsilon_2 \in \fr{a}^\ast$, and 
\begin{align}
\Sigma 
:= \{ \pm \alpha_1, \ \pm \alpha_2, \ \pm(\alpha_1+\alpha_2), \ \pm(\alpha_1+2\alpha_2) \} 
\subset \fr{a}^\ast . 
\end{align} 
For simplicity of the notations, we put 
\begin{align} 
U := \left(\begin{array}{cc} & 1 \\ -1 & \end{array} \right) , \quad 
V := \left(\begin{array}{cc} & 1 \\  1 & \end{array} \right) , 
\end{align} 
which will be used several times henceforth. 

\begin{Lem} \label{lem_rootsp} 
The restricted root system of $\fr{g}$ with respect to $\fr{a}$ coincides with $\Sigma$. 
In particular, we have the following$:$ 
\begin{align*} 
	\fr{g}_{\alpha_1} &= \Span \left\{
		\left( \begin{array}{c|c|c} U & V & \phantom{O} \\\hline V & U & \\\hline & & \end{array} \right)
                \right\}, \\
	\fr{g}_{\alpha_2} &= \left\{
		\left( \begin{array}{c|c|c} & & y \\\hline & & y \\\hline {}^t y & -{}^t y & \end{array} \right)
        	\mid
		y = \left(\begin{array}{ccc} 0 & \cdots & 0 \\ \ast & \cdots & \ast \end{array} \right) \in M_{2,n-2}(\R)
                \right\}, \\
	\fr{g}_{\alpha_1 + \alpha_2} &= \left\{
		\left( \begin{array}{c|c|c} & & z \\\hline & & z \\\hline {}^t z & -{}^t z & \end{array} \right)
        	\mid
		z = \left(\begin{array}{ccc} \ast & \cdots & \ast \\ 0 & \cdots & 0 \end{array} \right) \in M_{2,n-2}(\R)
		\right\}, \\
	\fr{g}_{\alpha_1+2\alpha_2} &= \Span \left\{
		\left( \begin{array}{c|c|c} -U & U & \phantom{O} \\\hline -U & U & \\\hline & & \end{array} \right)
                \right\},
\end{align*} 
where the size of the every block decomposition is $(2, 2, n-2)$. 
Therefore, both $\fr{g}_{\alpha_1}$ and $\fr{g}_{\alpha_1+2\alpha_2}$ are of dimension $1$, 
whereas both $\fr{g}_{\alpha_2}$ and $\fr{g}_{\alpha_1+\alpha_2}$ are of dimension $n-2$. 
\end{Lem}

\begin{proof}
One can show each inclusion $(\supset)$ by direct calculations. 
The converse inclusions can be seen by dimensional reasons. 
In fact, by the definition of the Iwasawa decomposition, one has 
\begin{align} 
\label{eq:root-space}
\fr{g}_{\alpha_1} \oplus \fr{g}_{\alpha_2} \oplus \fr{g}_{\alpha_1+\alpha_2} 
\oplus \fr{g}_{\alpha_1+2\alpha_2} \subset \LN , 
\end{align} 
under a suitable choice of the set of simple roots. 
Furthermore, since the corresponding group $S$ with Lie algebra $\fr{a} \oplus \fr{n}$ 
acts simply transitively on $G_2^\ast(\R^{n+2})$, 
we have 
\begin{align} 
2n = \dim G_2^\ast(\R^{n+2}) = \dim (\fr{a} \oplus \fr{n}) = 2 + \dim \fr{n} . 
\end{align} 
Therefore, all the equalities in the assertion, 
and also the equality in (\ref{eq:root-space}) hold. 
This concludes that the restricted root system coincides with $\Sigma$. 
\end{proof} 

Note that the root system $\Sigma$ is of type $B_2$, 
and $\Lambda := \{ \alpha_1 , \alpha_2 \}$ is a set of simple roots of $\Sigma$. 
We thus have 
\begin{align} 
\fr{a} \oplus \fr{n} = \fr{a} \oplus \fr{g}_{\alpha_1} \oplus \fr{g}_{\alpha_2} 
\oplus \fr{g}_{\alpha_1+\alpha_2} 
\oplus \fr{g}_{\alpha_1+2\alpha_2} . 
\end{align} 
We here take a basis of $\fr{a} \oplus \fr{n}$, 
and see that it is isomorphic to $\fr{s}(c)$ defined in Definition~\ref{solv.model}. 

\begin{Prop} 
\label{prop:basis} 
Let us define 
\begin{align}
	A_1 &:= (c/2) \left( (E_{13} + E_{31}) - (E_{24} + E_{42}) \right) \in \fr{a}, \\
	A_2 &:= (c/2) \left( (E_{13} + E_{31}) + (E_{24} + E_{42}) \right) \in \fr{a}, \\
	X_0 &:= (c/2) \left( \begin{array}{c|c|c}
        		U & V & \phantom{O} \\ \hline 
                        V & U &             \\ \hline 
                          &   & 
                      \end{array} \right) \in \fr{g}_{\alpha_1}, \\
	Y_i &:= (c/\sqrt{2}) \left( \begin{array}{c|c|c}
                                  &              & E_{2i} \\ \hline 
                                  &              & E_{2i} \\ \hline 
                      {}^t E_{2i} & -{}^t E_{2i} & 
                      \end{array} \right) \in \fr{g}_{\alpha_2}, \\
	Z_i &:= (c/\sqrt{2}) \left( \begin{array}{c|c|c}
                                  &              & E_{1i} \\ \hline 
                                  &              & E_{1i} \\ \hline 
                      {}^t E_{1i} & -{}^t E_{1i} & 
                      \end{array} \right) \in \fr{g}_{\alpha_1+\alpha_2},\\
	W_0 &:= (c/2) \left( \begin{array}{c|c|c}
        		-U & U & \phantom{O} \\ \hline 
                        -U & U &             \\ \hline 
                          &    & 
                      \end{array} \right) \in \fr{g}_{\alpha_1+2\alpha_2} . 
\end{align}
Then they form a basis of $\fr{a} \oplus \fr{n}$, and the bracket relations among them 
are exactly same as the ones in Definition~\ref{solv.model}. 
\end{Prop} 

\begin{proof} 
It follows immediately from Lemma~\ref{lem_rootsp} that they form a basis of $\fr{a} \oplus \fr{n}$. 
The bracket relations among them can be calculated directly, which we omit. 
\end{proof} 

\subsection{The solvable model of $G_2^\ast(\R^{n+2})$} 

In this subsection we prove Theorem~\ref{thm_main4}. 
Since the Lie algebras $\fr{a} \oplus \fr{n}$ and $\fr{s}(c)$ are isomorphic, 
we have only to show that the inner products and the complex structures are also isomorphic. 

First of all, 
we describe the metric and the complex structure on $G_2^\ast(\R^{n+2})$, 
in terms of the Cartan decomposition $\fr{g} = \fr{k} \oplus \fr{p}$. 

\begin{Prop} 
\label{prop:structure_grassmann} 
For $G_2^\ast(\R^{n+2})$ with $n \geq 3$, we have the following$:$ 
\begin{enumerate} 
\item[$(1)$] 
Let $c>0$, and define $\langle X,Y \rangle_{\fr{g}} := (1/c^2) \tr ({}^t X Y)$ for $X, Y \in \fr{g}$. 
Then $\langle , \rangle_{\fr{g}} |_{\fr{p} \times \fr{p}}$ gives a $G$-invariant metric on $G_2^\ast(\R^{n+2})$, 
and every $G$-invariant metric on $G_2^\ast(\R^{n+2})$ is obtained in this way. 
\item[$(2)$] 
The minimal sectional curvature of $G_2^\ast(\R^{n+2})$ 
with respect to the metric given by the above $\langle , \rangle_{\fr{g}} |_{\fr{p} \times \fr{p}}$ 
coincides with $-c^2$. 
\item[$(3)$] 
The following $Z \in \fr{k}$ gives a complex structure on $G_2^\ast(\R^{n+2})$, 
which is unique up to sign$:$ 
\begin{align*} 
Z :=  \left( \begin{array}{c|ccc}
		U & &   & \\\hline
		  & &   & \\
		  & & \phantom{Y}  & \\
		  & &   & 
		\end{array} \right) , \quad 
U := \left(\begin{array}{cc} & 1 \\ -1 & \end{array} \right) . 
\end{align*}
\end{enumerate} 
\end{Prop} 

\begin{proof} 
One knows that the Killing form $B$ of $\fr{g}$ can be expressed as 
\begin{align} 
B(X,Y) = k^\prime \tr (X Y) 
\end{align} 
for some $k^\prime > 0$. 
On the other hand, the Cartan involution $\theta$ satisfies 
\begin{align} 
\label{eq:theta} 
\theta(X) = I_{2,n} X I_{2,n} = (- {}^t X I_{2,n}) I_{2,n} = - {}^t X 
\end{align} 
for all $X \in \fr{g}$. 
Therefore, by putting $k := 1 / (k^\prime c^2)$, we have 
\begin{align}
(1/c^2) \tr ({}^t X Y) 
= (1/ (k^\prime c^2)) B({}^t X , Y) = - k B (\theta X , Y) . 
\end{align}
Hence Proposition~\ref{prop:structure_general} yields the first assertion. 

We show the second assertion. 
According to \cite[Theorem~11.1, Ch.~VII]{Hel}, 
the minimal sectional curvature of a symmetric space of noncompact type
is given by $- \langle H_{\tilde{\alpha}}, H_{\tilde{\alpha}} \rangle_{\fr{g}}$, 
where $\tilde{\alpha}$ is the highest root, and $H_{\tilde{\alpha}}$ is the root vector determined by 
$\langle H, H_{\tilde{\alpha}} \rangle_{\fr{g}} = \tilde{\alpha} (H)$ 
for every $H \in \fr{a}$. 
In the case of $G_2^\ast(\R^{n+2})$, 
we have $\tilde{\alpha} = \alpha_1 + 2 \alpha_2$, 
and it follows from direct calculations that 
\begin{align} 
H_{\tilde{\alpha}} 
	= (c^2/2) \left( (E_{13} + E_{31}) + (E_{24} + E_{42}) \right) . 
\end{align} 
Therefore, one has $\langle H_{\tilde{\alpha}}, H_{\tilde{\alpha}} \rangle_{\fr{g}} = c^2$, 
which shows the second assertion. 

For the third assertion, 
it is easy to see that the center $C(\fr{k})$ of $\fr{k}$ coincides with $\Span \left\{ Z \right\}$. 
Furthermore, $Z$ satisfies 
\begin{align} 
(\ad(Z)|_\fr{p})^2 = -\id_\fr{p} , 
\end{align} 
since $U^2=-I$. 
This completes the proof. 
\end{proof}

We are now in position to prove Theorem~\ref{thm_main4}, 
by studying the structures of the Lie algebra $\fr{a} \oplus \fr{n}$ explicitly. 

\begin{proof}[Proof of Theorem~\ref{thm_main4}] 
Let $(\fr{a} \oplus \fr{n} , \langle , \rangle , J)$ be the solvable model 
of $G_2^\ast(\R^{n+2})$, defined in Definition~\ref{def:model}. 
Consider the basis 
\begin{align} 
\{A_1, A_2, X_0, Y_1, \ldots, Y_{n-2}, Z_1, \ldots, Z_{n-2}, W_0 \} 
\end{align} 
given in Proposition~\ref{prop:basis}. 
We have only to show that 
this basis satisfies the conditions in Definition~\ref{solv.model}. 

Recall that $\langle X,Y \rangle_{\fr{g}} := (1/c^2) \tr ({}^t X Y)$ for $X, Y \in \fr{g}$. 
Then, by Proposition~\ref{prop:structure_grassmann}, 
the minimal sectional curvature with respect to $\langle , \rangle_{\fr{g}} |_{\fr{p} \times \fr{p}}$ 
is equal to $-c^2$. 
According to Proposition~\ref{prop:model_isomorphic}, 
the corresponding inner product $\langle, \rangle$ on $\fr{a} \oplus \fr{n}$ is given by 
\begin{align} 
\langle X , Y \rangle = 
(1/c^2) \tr (X_{\fr{a}} Y_{\fr{a}}) + (1/(2 c^2)) \tr ({}^t X_{\fr{n}} Y_{\fr{n}}) 
\end{align} 
for all $X, Y \in \fr{a} \oplus \fr{n}$. 
One can show by direct calculations that the above basis is orthonormal 
with respect to this inner product $\langle , \rangle$. 

Recall that the complex structure $J$ on $\fr{a}\oplus\fr{n}$ is defined by 
\begin{align} 
J := \pi^{-1} \circ \ad(Z) \circ \pi , 
\end{align} 
where $\pi : \fr{a} \oplus \fr{n} \to \fr{p}$ is the projection given by 
$\pi (X) = (1/2) (X - \theta X)$, 
and $Z$ is given in Proposition~\ref{prop:structure_grassmann}. 
Hence one can see that 
\begin{align} 
J (X_0) = \pi^{-1} [ 
\left( \begin{array}{c|c|c}
U & 0 & \phantom{O} \\ \hline 
0 & 0 &             \\ \hline 
&   & 
\end{array} \right) , 
(c/2) \left( \begin{array}{c|c|c}
0 & V & \phantom{O} \\ \hline 
V & 0 &             \\ \hline 
&   & 
\end{array} \right) 
] = A_1 . 
\end{align} 
By similar calculations, we have 
\begin{align} 
J(W_0) = - A_2 , \quad J(Y_i) = Z_i . 
\end{align} 
Therefore $J$ satisfies the condition in Definition~\ref{solv.model}. 
\end{proof}

\section{Proof of the main theorem} 

In this section, 
we prove our main theorem, 
stating that 
$G_{0,2}$ of dimension $2n+1$ with $n \geq 2$ 
can be realized as a homogeneous hypersurface in the noncompact real two-plane Grassmannian $G_2^\ast(\R^{n+3})$. 

\subsection{Preliminaries on Lie hypersurfaces} 

The contact metric manifolds $G_{0,2}$ will be realized as particular homogeneous hypersurfaces, 
called Lie hypersurfaces. 
For a Riemannian manifold $M$, 
an orbit of a cohomogeneity one action on $M$ without singular orbits 
is called a \textit{Lie hypersurface} in $M$. 
In this subsection, 
we review some results on Lie hypersurfaces in symmetric spaces of noncompact type. 

Let $M=G/K$ be an irreducible Riemannian symmetric space of noncompact type, 
and keep the same notations in previous sections. 
Note that Lie hypersurfaces in $M = G/K$ 
have been classified up to isometric congruence (\cite{MR2015258}), 
which we describe here. 
One needs the notation 
\begin{align} 
\LS_E := (\fr{a} \oplus \fr{n}) \ominus \R E \quad (E \in \fr{a} \oplus \fr{n}) ,  
\end{align} 
where $\ominus$ denotes the orthogonal complement. 
Note that, if $E$ is in a suitable place, then $\LS_E$ is a subalgebra of $\fr{a} \oplus \fr{n}$. 

\begin{Thm}[\cite{MR2015258}] \label{5-1}
Let $E \in \fr{a} \oplus \fr{n}$ be a nonzero vector, 
and assume that it satisfies $E \in \LA$ or $E \in \LG_{\alpha_i}$, 
where $\alpha_i \in \Lambda$ is a simple root. 
Then $\LS_E := (\fr{a} \oplus \fr{n}) \ominus \R E$ is a Lie subalgebra of $\fr{a} \oplus \fr{n}$, 
and furthermore, orbits of the corresponding connected Lie subgroup $S_E$ of $S$ 
are Lie hypersurfaces in $M$.
Conversely, every Lie hypersurface in $M$ can be constructed in this way up to isometric congruence. 
\end{Thm} 

The orbit $S_E.o$ through the origin $o$ can naturally be identified with the Lie group $S_E$ itself, 
equipped with the left-invariant metric induced from $S_E \subset S$. 
We here summarize some properties of the case of $E \in \LA$, 
which is of our particular interest. 

\begin{Prop} 
\label{prop:S_E} 
Let $E \in \LA$ be a nonzero unit vector. 
Then the action of $S_E$ satisfies the following$:$ 
\begin{enumerate} 
\item[$(1)$] 
All orbits of $S_E$ are diffeomorphic to Euclidean spaces. 
\item[$(2)$] 
All orbits of $S_E$ are isometrically congruent to each other. 
\item[$(3)$] 
All orbits of $S_E$ are solvsolitons. 
\end{enumerate} 
\end{Prop} 

\begin{proof} 
The first assertion follows from \cite[Proposition~1]{BB}, 
and the second assertion has been known in \cite{MR2015258}, see also \cite{KT}. 
We show the third assertion. 
Recall that 
the ambient space $M \cong S$ is Einstein with negative constant. 
It then follows from Proposition~\ref{prop:solvsoliton} that 
the Lie algebra $\fr{s}_E$ is a solvsoliton. 
Hence $S_E.o \cong S_E$ is also a solvsoliton, 
since it is simply-connected by the first assertion. 
By the second assertion, all orbits of $S_E$ are solvsolitons. 
\end{proof} 

\subsection{A particular Lie hypersurface in $G_2^\ast(\R^{n+3})$} 

In this subsection, 
we define and study a particular Lie hypersurface in 
the noncompact real two-plane Grassmannian $M = G_2^\ast(\R^{n+3})$. 
For the latter convenience, we equip $M$ with the metric given by $c = 2 \sqrt{2}$. 
By Theorem~\ref{thm_main4}, the minimal sectional curvature with respect to this metric is $-8$. 

First of all, we define our Lie hypersurface. 
Let $(\fr{s}(2\sqrt{2}) = \fr{a} \oplus \fr{n} , \langle, \rangle, J)$ 
be the solvable model of $G_2^\ast(\R^{n+3})$, defined in Section~\ref{sec:4}. 
Recall that 
\begin{align} 
\fr{a} = \Span \{ A_1, A_2 \} , \quad 
\fr{n} = \Span \{ X_0, Y_1, \ldots, Y_{n-1}, Z_1, \ldots, Z_{n-1}, W_0 \} . 
\end{align} 
Note that $\dim \fr{a} = 2$ and $\dim \fr{n} = 2n$, 
which agree with $\dim G_2^\ast(\R^{n+3}) = 2n+2$. 
We here take a unit vector 
\begin{align}
N := (1/\sqrt{2}) (- A_1-A_2) \in \fr{a}.
\end{align}
By Theorem~\ref{5-1}, one knows that 
$\fr{s}_{N} := (\fr{a}\oplus\fr{n}) \ominus \mathbb{R}N$ is a Lie subalgebra of $\fr{a} \oplus \fr{n}$. 
Note that $\dim \fr{s}_{N} = 2n+1$. 
Let us put 
\begin{align} 
\begin{aligned} 
		\xi       &:= -J N = (1/\sqrt{2}) (-X_0+W_0), \\
		\xi^\perp &:= (1/\sqrt{2}) (-X_0-W_0),\\
		T         &:= J \xi^{\perp} = (1/\sqrt{2}) (-A_1+A_2) . 
\end{aligned} 
\end{align} 
Then, we obtain
\begin{align}
\label{eq:basis} 
\begin{aligned}
 \fr{s}_N 
 &= \Span \{ T \} \oplus \fr{n} \\
 &= \Span \{ \xi, \xi^\perp, T, Y_1,\ldots, Y_{n-1}, Z_1,\ldots, Z_{n-1} \},
\end{aligned}
 \end{align}
which forms an orthonormal basis. 
The bracket relations on $\fr{s}_N$ can be seen from the definition of the solvable model, which are summarized as follows. 

\begin{Lem} 
\label{lem:bracket} 
The nonzero bracket relations of $\fr{s}_N$ with respect to the above basis are given by 
\begin{align*} 
& [\xi, T  ] = 2 \xi^\perp , \quad 
[\xi, Y_i] = -2 Z_i , \quad 
[\xi^\perp, T  ] = 2 \xi , \quad 
[\xi^\perp, Y_i] = -2 Z_i , \\ 
& [T, Y_i] = 2 Y_i , \quad 
[-Z_i, Y_i] = -2 \xi^\perp + 2 \xi . 
\end{align*} 
\end{Lem} 

Denote by ${S_N}$ the corresponding connected Lie subgroup of $S$. 
Then we obtain a Lie hypersurface $S_N.o$ in $G_2^\ast(\R^{n+3})$. 

In the remaining of this subsection, 
we show two properties of $S_N.o \cong S_N$. 
The first property is that $S_N.o \cong S_N$ is a contact metric manifold. 
Recall that $G_2^\ast(\R^{n+3})$ is K\"{a}hler, 
and hence its real hypersurface $S_N.o$ admits an almost contact metric structure. 

\begin{Prop}
\label{prop:S_Ncontact} 
The Lie hypersurface $S_N.o \cong S_N$ is a contact metric manifold 
with respect to the structures given in Proposition~\ref{ACMM}. 
\end{Prop}

\begin{proof}
Recall that 
the almost contact metric structures of $\fr{s}_N$ are given as follows$:$ 
\begin{itemize}
\item 
the metric $\langle, \rangle$ is the induced metric, 
so that the basis given in (\ref{eq:basis}) is orthonormal, 
\item 
the characteristic vector field is given by $-J N = \xi$,
\item 
the $1$-form $\eta$ is defined by $\eta(X) = \langle X, \xi \rangle$, 
\item 
the $(1,1)$-tensor field $\varphi$ is defined by $\varphi(X) = J X - \eta(X) N$. 
\end{itemize}
We show $\Phi = d \eta$, 
where $\Phi$ is the fundamental $2$-form defined by 
\begin{align} 
\Phi(X,Y) := \langle X , \varphi(Y) \rangle . 
\end{align} 
For $d \eta$, since the structures are left-invariant, one knows that 
\begin{align} 
2 d \eta(X,Y) = - \eta([X,Y]) = - \langle [X , Y] , \xi \rangle . 
\end{align} 
We thus have only to show that 
\begin{align} 
\label{eq:claim5-6} 
2 \langle X , \varphi(Y) \rangle = - \langle [X , Y] , \xi \rangle 
\end{align} 
for all 
$X, Y \in \fr{s}_N$. 
By direct calculations, 
one can see that $\varphi$ satisfies 
\begin{align} 
	\varphi(\xi)=0, \ 
        \varphi(\xi^\perp) = T, \ 
\varphi(T)=-\xi^{\perp} , \ 
\varphi(Y_i)=-Z_i , \ 
        \varphi(Z_i) = - Y_i. 
\end{align}
Therefore, in terms of the bracket relations given in Lemma~\ref{lem:bracket}, 
one can directly check that (\ref{eq:claim5-6}) holds, which completes the proof. 
\end{proof}

\begin{Remark}
Although we gave a proof of Proposition~\ref{prop:S_Ncontact} by direct calculations, 
we have to point out that the result itself has been known. 
In fact, 
Berndt and Suh have classified contact real 
hypersurfaces with constant mean curvature in $G_2^\ast(\R^{n+3})$ (\cite[Theorem~1.2]{MR3326043}), 
and the Lie hypersurface $S_N.o$ is contained in their list. 
In their paper, 
$S_N.o$ is called a horosphere whose center at infinity is the equivalence class of 
an $\mathcal{A}$-principal geodesic in $G_2^\ast(\R^{n+3})$. 
\end{Remark}

The second property which we show in this subsection is that 
$S_N.o \cong S_N$ is a nongradient expanding Ricci soliton. 
In fact, it is a nontrivial solvsoliton. 

\begin{Prop} 
\label{prop:5-6}
The Lie hypersurface $S_N.o \cong S_N$ is a nongradient expanding Ricci soliton. 
\end{Prop}

\begin{proof} 
We know that $S_N.o \cong S_N$ is a solvsoliton by Proposition~\ref{prop:S_E}. 
Thus, in view of Proposition~\ref{prop:L11}, 
we have only to show that this solvsoliton is nontrivial (that is, not Einstein). 
Recall that $\fr{s}_N$ is a subalgebra of $\fr{s}(2 \sqrt{2}) = \fr{a} \oplus \fr{n}$, 
given by 
\begin{align} 
\fr{s}_N = \fr{a}^\prime \oplus \fr{n} \quad 
(\fr{a}^\prime := \Span \{ T \} \subset \fr{a}) . 
\end{align} 
In terms of the bracket relations of the solvable model, 
one can directly show that the mean curvature vector $H_0$ of $\fr{s}(2 \sqrt{2})$ is given by 
\begin{align} 
H_0 = 2 \sqrt{2} A_1 +  2 \sqrt{2} (n-1) A_2 . 
\end{align} 
Since $T = (1/\sqrt{2})(-A_1+A_2)$ and $n \geq 2$, 
we have $H_0 \not \in \fr{a}^\prime$. 
It hence follows from Proposition~\ref{thm:heber} 
that $\fr{s}_N$ is not Einstein, which completes the proof. 
\end{proof}

\subsection{Constructing an isomorphism} 

We are now in the position to prove our main theorem. 
For the proof, 
we review the left-invariant contact metric Lie group $G_{0,2}$, and its Lie algebra 
\begin{align}
 \fr{g}_{0,2} = \Span \{\hat{\xi}, \hat{X}_1, \ldots, \hat{X}_n, \hat{Y}_1, \ldots, \hat{Y}_n\}.
\end{align}
Note that $\dim \fr{g}_{0,2} = 2n+1$.
By substituting $(\alpha, \beta)=(0,2)$ into Definition~\ref{def:gab}, 
one can see that the nonzero bracket relations on $\fr{g}_{0,2}$ are given by 
\begin{align} 
\begin{aligned}{} 
        [\hat{\xi}, \hat{Y}_i] &= 2 \hat{X}_i & (i \geq 1) , \\
        [\hat{Y}_2, \hat{Y}_i] &= 2 \hat{Y}_i & (i \ne 2) , \\
        [\hat{X}_2, \hat{Y}_2] &= 2 \hat{\xi}, \\
        [\hat{X}_2, \hat{Y}_i] &= 2 \hat{X}_i & (i \ne 2) , \\
        [\hat{X}_i, \hat{Y}_i] &= -2 \hat{X}_2+ 2\hat{\xi} & (i \ne 2) . 
\end{aligned} 
\end{align} 
The left-invariant contact metric structures on $G_{0,2}$ are given as follows$:$ 
\begin{itemize}
\item the metric $\langle, \rangle$ is defined so that the above basis is orthonormal,
\item the characteristic vector field $\xi$ is given by $\hat{\xi}$,
\item the $1$-form $\eta$ is the metric dual of $\hat{\xi}$, that is, $\eta(X) = \langle X, \hat{\xi} \rangle$, 
\item the $(1,1)$-tensor field $\varphi$ is defined by
\begin{align*}
\varphi(\hat{\xi})=0, \quad \varphi(\hat{X}_i) = \hat{Y}_i, \quad \varphi(\hat{Y}_i)=-\hat{X}_i . 
\end{align*}
\end{itemize}

\begin{Thm}
\label{thm:5-7} 
The contact metric manifold $G_{0,2}$ is isomorphic to the Lie hypersurface $S_N.o$ in $G_2^\ast(\R^{n+3})$ 
as contact metric manifolds. 
\end{Thm}

\begin{proof}
Consider the following correspondence from $\fr{s}_N$ to $\fr{g}_{0,2}$:
\begin{align} 
\begin{aligned} 
		 \xi  		&\mapsto \hat{\xi}, 	& \\
 	        -Z_1 		&\mapsto \hat{X}_1, 	& \\
		 \xi^\perp 	&\mapsto \hat{X}_2, 	& \\
        	-Z_i 		&\mapsto \hat{X}_{i+1} & (i \ge 2) , \\
		 Y_1 		&\mapsto \hat{Y}_1, 	& \\
		 T   		&\mapsto \hat{Y}_2,	& \\
		 Y_i 		&\mapsto \hat{Y}_{i+1}	& (i \ge 2) . 
\end{aligned} 
\end{align} 
By comparing the bracket relations on $\fr{s}_N$ and $\fr{g}_{0,2}$, 
one can see that the above correspondence gives an isomorphism from $\fr{s}_N$ to $\fr{g}_{0,2}$ as Lie algebras. 
Since both $S_N$ and $G_{0,2}$ are simply-connected, 
this gives rise to an isomorphism between $S_N$ and $G_{0,2}$ as Lie groups. 
By the definitions of the contact metric structures on $S_N$ and $G_{0,2}$, 
one can see that they are isomorphic as contact metric manifolds. 
\end{proof} 

Therefore, Theorem~\ref{thm} directly follows from Theorem~\ref{thm:5-7}. 
Furthermore, by combining with Proposition~\ref{prop:5-6}, 
we can conclude that $G_{0,2}$ is a nongradient expanding Ricci soliton, 
which completes the proof of Corollary~\ref{cor}.

\end{document}